\documentclass{amsart}

\usepackage[T1]{fontenc}
\usepackage{amsmath}
\usepackage{amsthm}
\usepackage{amsfonts}
\usepackage{amssymb}
\usepackage{enumerate}
\usepackage{mathrsfs}
\usepackage{graphicx}
\usepackage[labelfont=rm]{caption}
\usepackage{subcaption}
\usepackage{tikz}
\usepackage{defs}
\usepackage[hidelinks]{hyperref}

\newif\iftikziii
\tikziiitrue

\iftikziii
\usetikzlibrary{matrix,calc,arrows.meta,knots,intersections,decorations.markings,cd}
\else
\usetikzlibrary{matrix,calc,knots,intersections,decorations.markings}
\fi

\title{Wirtinger curves, Artin groups, and hypocycloids}
\author[E. Artal]{Enrique Artal Bartolo}

\author[J.I. Cogolludo]{Jos{\'e} Ignacio Cogolludo-Agust{\'i}n}
\address{Departamento de Matem\'aticas, IUMA\\ 
Universidad de Zaragoza\\ 
C.~Pedro Cerbuna 12\\ 
50009 Zaragoza, Spain} 
\email{artal@unizar.es,jicogo@unizar.es}

\author[J.~Mart\'{i}n]{Jorge Mart\'{i}n-Morales}
\address{Centro Universitario de la Defensa-IUMA \\
Academia General Militar \\
Ctra.~de Huesca s/n. 50090, Zaragoza, Spain}
\email{jorge@unizar.es}
\urladdr{\url{http://cud.unizar.es/martin}}

\dedicatory{Con cari{\~n}o para nuestra maestra y compa{\~nera} Maite Lozano}

\thanks{Partially supported by MTM2016-76868-C2-2-P
and Grupo Geometr{\'i}a of Gobierno de Arag{\'o}n/Fondo 
Social Europeo.}  

\subjclass[2010]{14H30, 57M10, 32S05, 20F36, 14H50}  

\keywords{Fundamental group, Zariski-van Kampen, Artin groups}

\makeindex

\begin{document}

\begin{abstract}
The computation of the fundamental group of the complement
of an algebraic plane curve has been theoretically solved
since Zariski-van Kampen, but actual computations are usually
cumbersome. In this work, we describe the notion of Wirtinger
presentation of such a group relying on the real picture
of the curve and with the same combinatorial flavor as the
classical Wirtinger presentation; we determine a significant family of curves
for which Wirtinger presentation provides the required fundamental group.
The above methods allow us to compute that fundamental group for an
infinite subfamily of hypocycloids, relating them with Artin groups.
\end{abstract}

\maketitle

\section*{Introduction}

In \cite{wirt:05,wirt:27}, W.~Wirtinger introduced his well-known method to compute the
fundamental group of the complement of a knot. His primary aim was to apply this
method to algebraic knots and links~\cite{br:28}, i.e., links obtained as the transversal intersection
of an algebraic curve (in~$\bc^2$) with a \emph{small} sphere centered at a singular point.
His method also works for any link and it is most useful for such computations.
One of its interesting features is that it provides a simple combinatorial method
to compute this group from the diagram of a knot or link, while keeping track of its
geometrical definition. 
The other practical method to compute this group comes
from Artin braid groups~\cite{art:26,art:47}; 
this is the idea behind Zariski-van Kampen's
method~\cite{zr:29,vk:33} in order to give a presentation of the fundamental group of the global complement
of a plane algebraic curve (later formalized as braid monodromy by Moishezon~\cite{mz:83}).

In this paper, we are going to adapt Wirtinger's method to compute the fundamental group of the complement
of some plane algebraic curves. These curves must have a real equation and a \emph{rich} algebraic picture.
Our goal is to provide (for a relatively small though significant family of curves) a combinatorial
method for the computation of this group. When implemented, Zariski-van Kampen's method tends to rely 
generically on heavy 
numerical computations using floating-point arithmetics (see~\cite{mbrr:16} for a reasonably efficient implementation
in \texttt{Sagemath} that assures an exact output), this is why some theoretical methods applying to infinite 
families of curves are needed.

The origin of our interest in this method started in~\cite{ac},
where the fundamental group of the complement of some small-degree complexified hypocycloids was studied. 
Our techniques were applied to curves where the real picture gave a lot of information. This is not the case 
for the whole family of
hypocycloids, but the use of their symmetries allowed us to use similar techniques 
for their resulting quotients whenever a rich real picture was obtained.

Following the notation introduced in~\cite{ac}, a hypocycloid is determined by two integers $0<\ell<k$, $\gcd(k,\ell)=1$ 
as the real curve traced by a fixed point on a circumference of radius~$\ell$ while rolling inside a circumference of radius~$N:=k+\ell$. 
Such real curves admit a real algebraic equation. The fundamental group of the complement of their complexified 
version $\cC_{k,\ell}$ is the focus of our interest here.

Some partial results treated in~\cite{ac} and all the hypocycloids for $N\leq 11$ (using the \texttt{Sirocco} 
package by Marco and Rodr{\'i}guez~\cite{mbrr:16}) lead to the following conjecture describing them as Artin groups.
We must emphasize that there is no hope that the techniques presented in~\cite{ac} as well as the computational methods 
used in small degrees may be generalized to the whole family.

\begin{conj}
\label{conj:hypocycloids}
For any pair of coprime integers $0<\ell<k$, $N:=k+\ell$, the fundamental group
for $\cC_{k,\ell}$ is the Artin group of the $N$-gon.
\end{conj}

In a forthcoming paper the conjecture will be proved for hypocycloids of type $(k,k-2)$ ($k$~odd), using quotient singularities.

The paper is organized as follows: in section~\ref{subsec:zvk} the original Zariski-van Kampen method is recalled as way to provide
a presentation for the fundamental group of the complement of an affine complex curve $\cC$. This presentation has an extra property
proved by A.Libgober in~\cite{Libgober-homotopytype} which assures that the homotopy type of this complement coincides with 
that of the CW-complex associated with the given presentation. Section~\ref{sec:wirtinger} is devoted to introducing the concept
of curves of Wirtinger type as a complexified real curve satisfying certain properties with respect to a projection. Associated
with the real picture of such curves one can define a finite presentation \emph{\`a la Wirtinger}. The resulting group is not 
necessarily isomorphic to the expected group $\pi_1(\bc^2\setminus \cC)$ as illustrated by a series of examples. However, under 
certain additional conditions they are indeed isomorphic. This is stated in the main Theorem~\ref{thm:main0}. The proof of this 
result is given in section~\ref{sec:main-thm}, where the Zariski-van Kampen presentation is transformed into the Wirtinger
presentation preserving the homotopy type of the associated CW-complex in Corollary~\ref{cor:hypocycloids}. 
The purpose of section~\ref{sec:hypocycloids} is to state and prove Conjecture~\ref{conj:hypocycloids} for $\ell=k-1$ 
--\,see Corollary~\ref{cor:hypocycloids}. A brief discussion on Artin groups and first properties of hypocycloids completes the 
section. Finally, a series of examples in section~\ref{sec:wirtingerext} exhibit how the conditions of 
Theorem~\ref{thm:main0} can be relaxed at the cost of understanding the so-called obstruction points. This idea results in a more
sophisticated version of the Wirtinger presentation, whose description goes beyond the scope of this paper and will be presented
somewhere else.

\section{The Zariski-van Kampen method}\label{subsec:zvk}

Let $\cC\subset\bc^2$ be a plane algebraic curve. We assume that for a given coordinate system the equation
of $\cC$ is given by a polynomial $f(x,y)\in\bc[x,y]$ such that $\deg_y f=d$ and the coefficient
of $f$ in $y^d$ as polynomial in~$\bc[x][y]$ is~$1$. 
As we are only interested in the zero locus, we can assume $\cC$ to be reduced, i.e., $f$ is a square-free polynomial.
In particular $D(x):=\disc_y(f)\in\bc[x]$ is a non-zero polynomial.

The geometrical characterization for $f$ being monic in~$y$ is that $\cC$ contains neither a vertical line nor a 
vertical asymptote. By a \emph{vertical asymptote} we mean a vertical line that is tangent to the curve at infinity.
Let us consider $p:\bc^2\to \bc$ be the vertical projection $(x,y)\mapsto x$; the restriction
$p_{|\cC}$ fails to be a covering only at the points of $\Delta:=\{t\in\bc\mid D(t)=0\}$.
As a consequence, if the vertical line $x=t$ is denoted by $L_t$,
\[
p_{|}:\bc^2\setminus\left(\cC\cup\bigcup_{t\in\Delta}L_t\right)\to\bc\setminus\Delta
\]
is a locally trivial fibration. Let us denote by $r$ the cardinality of the discriminant $\Delta$.
Providing a suitable section of this fibration (over a big 
enough closed disk) (e.g. using some horizontal line), the following theorem holds.

\begin{thm}
Under the above hypotheses,
\[
\pi_1\left(\bc^2\setminus\left(\cC\cup\bigcup_{t\in\Delta}L_t\right);(x_0,y_0)\right)=
\left\langle
\mu_1,\dots,\mu_d,\alpha_1,\dots,\alpha_r
\left|
\underset{\quad 1\leq i\leq d,\quad 1\leq j\leq r}{\alpha_j^{-1}\cdot\mu_i\cdot\alpha_j=\mu_i^{\tau_j}}
\right.
\right\rangle.
\] 
\end{thm}
The loops $\mu_j$ correspond to a \emph{geometric} basis of the free group $\mathbb{F}_d:=\pi_1(L_{x_0}\setminus\cC;(x_0,y_0))$ (i.e., each element
is a meridian and the reversed product is homotopic to the boundary of
a big disk, see~\cite{mz:83,acc:01a});
the loops $\alpha_i$ correspond to the lift to the horizontal line $y=y_0$  of a \emph{geometric} basis of the 
free group $\pi_1(\bc\setminus\Delta;x_0)$. By the continuity of roots, these loops, provide
braids $\tau_j\in\mathbb{B}_d$ and the right action in the statement corresponds to the standard
right action of $\mathbb{B}_d$ on $\mathbb{F}_d$. 
We identify the braid group with the fundamental group of $(\CC^d\setminus \mathcal D)/\Sigma_d$ with base point 
$p_{|\cC}^{-1}(x_0)\subset \CC$, where $\mathcal D=\{(x_1,...,x_d)\in \CC^d\mid x_i=x_j \text{ for some } i<j\}$ and 
the quotient is given by the group action of the permutation group $\Sigma_d$ acting on the coordinates
$\sigma\cdot (x_1,...,x_d)=(x_{\sigma(1)},...,x_{\sigma(d)})$~\cite{art:26,art:47}.
The braid group is generated
by the standard \emph{half-twists}
$\sigma_1,\dots,\sigma_{d-1}$ and the action on the free
group is defined  by
\begin{equation}
\label{eq:braid-action}
\mu_i^{\sigma_j}:=
\begin{cases}
\mu_i&\text{ if }j\neq i,i-1\\
\mu_{i+1}&\text{ if }j=i\\
\mu_{i}\cdot\mu_{i-1}\cdot\mu_i^{-1}&\text{ if }j=i-1,
\end{cases} \qquad
\mu_i^{\sigma_j^{-1}}=
\begin{cases}
\mu_i&\text{ if }j\neq i,i-1\\
\mu_{i}^{-1}\cdot\mu_{i+1}\cdot\mu_i&\text{ if }j=i\\
\mu_{i-1}&\text{ if }j=i-1.
\end{cases}
\end{equation}
For the sake of completeness, the action of the inverse of the standard half-twists have been added.

Assume for simplicity that for each $t\in\Delta$, $p_{|\cC}$ fails only at one point $(t,y(t))$ to be a covering over $t$.
The main ideas behind the Zariski-van Kampen method are the following ones. On one side, $\alpha_j$ will be
null-homotopic in $\bc^2\setminus\cC$; on the other side, if the braid $\tau_j$ correspond to
$t_j\in\Delta$ (denoting $y_j:=y(t_j)$), then $\tau_j$ can be written as $\eta_j^{-1}\cdot\delta_j\cdot\eta_j$, where 
$\delta_j$ is a positive braid involving only the $m_j$ strings close to $(t_j,y_j)$ (and whose conjugacy class is determined 
by the topological type of $\cC$ at $(t_j,y_j)$). Note that, without loss of generality, one might assume that exactly the first 
$m_j$ strings are involved in~$\delta_j$.
\begin{cor}\label{cor:zvk}
\[
\pi_1\left(\bc^2\setminus\cC;(x_0,y_0)\right)=
\langle
\mu_1,\dots,\mu_d
\mid
\mu_i^{\delta_j\cdot\eta_j}=\mu_i^{\eta_j},\quad 1\leq j\leq r,\quad 1\leq i< m_j
\rangle.
\] 
Moreover, in case $\deg f=\deg_y f=d$, a presentation for the fundamental group
of the complement of the Zariski closure of $\cC$ in $\bp^2$ is obtained by adding the relation
\[
\mu_d\cdot\ldots\cdot\mu_1=1.
\]
\end{cor}

\begin{obs}
\label{rem:zvkrels}
The fact that the only needed relations come from $1\leq j<m_j$ is due to the fact that
$(\mu_{m_j}\cdot\ldots\cdot\mu_1)^{\delta_j}=\mu_{m_j}\cdot\ldots\cdot\mu_1$ and
$\mu_i^{\delta_j}=\mu_i$ for $m_j<i\leq d$.
Note that for double points, the corresponding $m_j$ equals 2, and hence only one relation 
is required. For meridians $x_1,x_2$ close to the singular point of $p_{|\cC}$, the following
relation is satisfied:
\[
x_1=x_2\ (m=0),\quad [x_1,x_2]=1\ (m=1),\quad x_1\cdot x_2\cdot x_1=x_2\cdot x_1\cdot x_2\ (m=2).
\]
This comes from the action of the local braid $\sigma_1^{m+1}$ on $x_1,x_2$ as described in~\eqref{eq:braid-action}.
\end{obs}

\section{Wirtinger curves}\label{sec:wirtinger}

In this section the new concept of Wirtinger curves is defined. For this class of curves, a presentation of 
the fundamental group of their complement can be combinatorially obtained from their real picture.
We start first by defining curves of \emph{Wirtinger type}, which will be those which are candidates
to be Wirtinger curves.

An affine curve $\mathcal{C}\subset\mathbb{C}^2$ 
is called of \emph{Wirtinger type} if:
\begin{enumerate}
\enet{(W\arabic{enumi})} 
 \item the ramification points of the vertical projection $p(x,y)=x$ are all real, that is,
 $$R=\{P=(t,y_t)\in \cC \mid p^{-1}(t) \text{ and } \cC \text{ do not intersect transversally at } P\}\subset \br^2,$$
 in particular, $\Delta=p(R)\subset\br$;
 \item the local branches of $\cC$ at $P\in R$ are all real,
 \item\label{W3} the vertical fiber $L=p^{-1}(x_0)$ of $p$ at some $x_0\in \br$ intersects the real part of $\cC$ with 
 maximal cardinality, that is, $\# (p|_\cC^{-1}(x_0)\cap \br^2)=\deg_y f$ 
 --\,note that this is not necessarily the degree of the curve,
 \item it contains no vertical asymptotes and no vertical lines, and 
  simple tangencies at smooth points are the only vertical lines in the tangent cone of $\cC$ at any point.
 \item the only singularities of $\cC$ are either double --\,type $\ba_m$-- or ordinary (i.e., smooth branches with pairwise distinct tangent cones).
\end{enumerate}

Given a curve $\cC$ of Wirtinger type, we consider its \emph{diagram} $\cC_\br=\cC\cap \br^2$. 
Denote by $\mathcal{V}_\mathcal{C}$ the set of singular points of $\mathcal{C}$, that is, 
$\mathcal{V}_\mathcal{C}=\Sing \cC=\br^2\cap \Sing \cC$; 
they will be referred to as the \emph{vertices} of $\mathcal{C}$.
The \emph{edges} of $\mathcal{C}$ are the closures of the connected components of 
$\mathcal{C}_{\mathbb{R}}\setminus\mathcal{V}_\mathcal{C}$. The set of edges will be denoted as $E_{\mathcal{C}}$.
Our purpose is to describe
a presentation of a group based on the vertices and edges of $\cC_\br$.
A \emph{Wirtinger presentation} $\G_\cC$ associated with a curve of Wirtinger type $\cC$ is given by a generating 
system parametrized by $E_\cC$, that is, $\{x_\ell\mid \ell\in E_\cC\}$. In addition, to each $P\in\cV_\cC$ the following 
relations are associated:
\begin{enumerate}
\enet{(R\arabic{enumi})} 
\item\label{R1}
\label{case:ordinary}
If $P$ is an ordinary real singular point of multiplicity~$m$ as in~\eqref{eq:ordinary},
then the edges associated with $P$ can be sorted out in two groups $\{x_1,\dots,x_m\}$ and $\{y_1,\dots,y_m\}$
as shown below. Define $\bar x_k=x_kx_{k-1}\cdots x_{1}$ and $\bar y_k=y_1y_{2}\cdots y_k$. Then

\begin{equation}
\label{eq:ordinary}
\begin{tabular}{cc}
\begin{tikzpicture}[yscale=.5,baseline=0]
\coordinate (O) at (0,0);
\foreach \x in {-2,...,2}
{
\coordinate (P\x) at (1,{\x});
\coordinate (Q\x) at (-1,{-\x});
}
\foreach \x in {-1,-2}
{
\draw (P\x)--(Q\x);
\pgfmathtruncatemacro{\z}{\x/2+5/2}
\node[right] at (P\x) {$y_\z$};
\node[left] at (Q\x) {$x_\z$};
}
\foreach \x in {1,2}
{
\draw (P\x)--(Q\x);
}
\node[right] at (P1) {$y_{m-1}$};
\node[left] at (Q1) {$x_{m-1}$};
\node[right] at (P2) {$y_{m}$};
\node[left] at (Q2) {$x_{m}$};

\node at (P0) {$\vdots$};
\node at (Q0) {$\vdots$};
\node[above=4pt] at (O) {$P$};
\end{tikzpicture}
&
$
\begin{cases}
[\bar x_m,x_j]=1\\
y_j=\bar x_j^{-1}\cdot x_j\cdot \bar x_j=\bar x_{j-1}^{-1}\cdot x_j\cdot \bar x_{j-1}
\end{cases}
$
\end{tabular}
\end{equation}
Note that $\bar x_m=\bar y_m$ is a consequence and the relations could be written also from left to right.

\item\label{R2}
\label{case:double}
If $P$ is of type $\ba_{m}$, then the following relations are added:

\begin{equation}
\label{eq:double}
\begin{tabular}{lcl}
\begin{tikzpicture}[baseline=0]
\begin{scope}[shift={(0,0)}]
\coordinate (O) at (0,0);
\coordinate (P) at (2,1);
\coordinate (Q) at (2,-1);
\node[above=2pt] at (O) {$P$};
\draw (P) to [out=-135,in=0] (O) to [out=0,in=135] (Q);
\node[right=2pt] at (P) {$x_1$};
\node[right=2pt] at (Q) {$x_2$};
\node[below=2pt] at (O) {$m={2k}$};
\end{scope}
\end{tikzpicture}
&&
$x_1 (x_2x_1)^k= (x_2x_1)^{k} x_2,$
\\
\begin{tikzpicture}[baseline=0]
\begin{scope}[shift={(1,1.5)}]
\coordinate (O) at (0,0);
\coordinate (P) at (2,1);
\coordinate (Q) at (2,-1);
\coordinate (R) at (-2,1);
\coordinate (S) at (-2,-1);
\node[above=2pt] at (O) {$P$};
\draw[line width=1.5pt] (P) to [out=-135,in=0] (O) to [out=180,in=-45] (R);
\draw (Q) to [out=135,in=0] (O) to [out=180,in=45] (S);
\node[right=2pt] at (P) {$y_1$};
\node[right=2pt] at (Q) {$y_2$};
\node[left=2pt] at (R) {$x_1$};
\node[left=2pt] at (S) {$x_2$};
\node[below=2pt] at (O) {$m={4k-1}$};
\end{scope}
\begin{scope}[shift={(1,-1.5)}]
\coordinate (O) at (0,0);
\coordinate (P) at (2,1);
\coordinate (Q) at (2,-1);
\coordinate (R) at (-2,1);
\coordinate (S) at (-2,-1);
\node[above=2pt] at (O) {$P$};
\draw (P) to [out=-135,in=0] (O) to [out=180,in=45] (S);
\draw[line width=1.5pt] (Q) to [out=135,in=0] (O) to [out=180,in=-45] (R);
\node[right=2pt] at (P) {$y_2$};
\node[right=2pt] at (Q) {$y_1$};
\node[left=2pt] at (R) {$x_1$};
\node[left=2pt] at (S) {$x_2$};
\node[below=2pt] at (O) {$m={4k+1}$};
\end{scope}
\end{tikzpicture}
&&
$
\begin{cases}
(x_1x_2)^{\frac{m+1}{2}}= (x_2x_1)^{\frac{m+1}{2}},&\\
y_i=(x_2x_1)^{-k} x_i (x_2x_1)^{k}, & i=1,2
\end{cases}
$
\\
\end{tabular}
\end{equation}
\end{enumerate}

\begin{obs}
\label{rem:xrels}
Note that relations in~\eqref{eq:ordinary} and \eqref{eq:double} involving only $x_i$'s correspond 
with the local braid-monodromy relations described in Remark~\ref{rem:zvkrels}.

The remaining relations describe meridians on one side of the singularity in terms of meridians on the 
other side as elements of the local fundamental group of the singular point,
whenever there are real branches on both sides of the singular point.
For a $\ba_{4k-1}$-singularity, the local braid monodromy is given by $\sigma_1^{4k}$,
hence the relation $y_i=(x_2x_1)^{-k} x_i (x_2x_1)^{k}$, $i=1,2$ in~\eqref{eq:double} is nothing but
$y_i=x_i^{\sigma_1^{-2k}}$ as described in~\eqref{eq:braid-action}.
For an ordinary singular point of multiplicity~$m$, the local braid monodromy
is given by $\Delta_m^2$ (where $\Delta_m$ represents a half-full twist in $m$~strands); the second line of~\eqref{eq:ordinary} is nothing but
$y_i=x_i^{{\Delta_m}^{-1}}$.
\end{obs}

\begin{obs}
Note that the case $m=1$ is in both the family of ordinary points
and double points. One can check that relations~\eqref{eq:ordinary} become
$$
[x_2x_1,x_1]=1, y_1=x_1, y_2=x_1^{-1}x_2x_1=x_2.
$$
whereas relations~\eqref{eq:double} become
$$
x_1x_2=x_2x_1, y_1=x_1, y_2=x_2.
$$
Therefore both sets of relations are trivially equivalent to
$$[x_1,x_2]=1, y_1=x_1,y_2=x_2.$$
In certain cases, $\G_\cC$ is a presentation of $\pi_1(\bc^2\setminus\cC)$, but not necessarily.
\end{obs}

\begin{ejm}
For any real smooth curve $\cC$ of Wirtinger type, note that $\G_\cC$ is a presentation of the free group of rank $r$, 
where $r$ is the number of connected components of $\cC_\br$, however $\pi_1(\bc^2\setminus\cC)=\bz$. 
It applies to $\cC:y^2-x(x^2-1)=0$ ($r=2$).
\end{ejm}

\begin{ejm}
Let $\cC$ be a strongly real line arrangement (with no vertical lines), that is, a finite union of lines where each 
line has a real equation. In particular $\cC=\{\ell_1\cdot\ldots\cdot \ell_r=0\}$ where $\ell_i\in \br[x,y]$,
$i=1,\dots,r$ are pairwise non-proportional linear forms. 
Note that in this case $\G_\cC$ gives the Salvetti presentation~\cite{sal:88}
of $\pi_1(\bc^2\setminus\cC)$ and it can be reduced to the
Zariski-van Kampen presentation associated with the vertical projection,
with Tietze transformations of type I and IIa.
\end{ejm}

\begin{ejm}
Consider the affine tricuspidal quartic whose line at infinity is bitangent (the \emph{deltoid}). This curve has a real equation
$$\cC:=\{(x,y)\in \bc^2 \mid 3(x^2 + y^2 )^2 + 24x(x^2 + y^2 ) + 6(x^2 + y^2 )-32x^3 - 1 = 0\}.$$
Its diagram is a triangle whose vertices are the three cusps. Therefore 
$\G_\cC=\langle x_1,x_2,x_3:x_1x_2x_1=x_2x_1x_2, x_2x_3x_2=x_3x_2x_3, x_3x_1x_3=x_1x_3x_1\rangle$
is the standard presentation of the Artin group of the triangle, which coincides with $\pi_1(\bc^2\setminus \cC)$ 
--\,see~\cite{Oka-tangential,ac}.
\end{ejm}

The following result offers a wide collection of examples of curves of Wirtinger type whose 
Wirtinger presentation is a presentation of the fundamental group~$\pi_1(\bc^2\setminus \cC)$.
In order to state the conditions one needs to introduce the simple concept of real branches \emph{facing} a vertical line.
Consider a vertical real line $L_\br$ and a singular point of the vertical projection $P\notin L_\br$. The vertical line through
$P$ separates the real plane in two half-planes, one of them say $H^+$ containing $L_\br$. If $H^+$ contains real branches at 
$P$, then these branches at $P$ are said to \emph{face}~$L_\br$.

\begin{thm}
\label{thm:main0}
Let $\mathcal{C}$ be a curve of Wirtinger type and such that the real part
of each irreducible component is connected. Let $L=p^{-1}(x_0)$ be a line satisfying{\rm~\ref{W3}} 
and let $B\subset \br^2$ be a closed topological disk (with piecewise smooth boundary) such that:
\begin{enumerate}
\enet{\rm(\arabic{enumi})} 
\item\label{thm:main1} $B\cup\cC_\br\cup L_\br$ is simply connected.
\item\label{thm:main2} There is a parallel \emph{real} plane 
$H_\varepsilon=\br\times (\br+\varepsilon \sqrt{-1})$ to $\br^2$ with $\varepsilon\neq 0$ such that 
\begin{equation}
\label{eq:Be}
B_\varepsilon\cap \cC=\emptyset, \text{ where } 
B_\varepsilon=\{(x,y+\varepsilon \sqrt{-1})\in H_\varepsilon \mid (x,y)\in B\}\subset H_\varepsilon.
\end{equation}
\item\label{thm:main3} All singularities of the vertical projection face~$L_\br$.
\end{enumerate}
Then $\G_\cC$ is a presentation of~$\pi_1(\bc^2\setminus \cC)$.
\end{thm}

\begin{obs}\label{obs:local}
Given a curve of Wirtinger type, the set of points $H_\varepsilon\cap \cC\subset H_\varepsilon$ projected
onto $\br^2$ via the real-part map will be referred to as \emph{obstruction points}.
Before we prove this result, we will describe strategies to determine the position of the obstruction
points based on the real picture $\cC_\br$ in order to check property~\eqref{eq:Be} without actual
computations.
First note that a smooth branch transversal to the vertical line can be locally parametrized by $y=0$
after a real change of coordinates. Therefore a parallel plane $H_\varepsilon$ near $P=(0,0)$ will 
be locally disjoint to $\cC$, since $y=v+\varepsilon \sqrt{-1}=0$ has no solution for $v\in \br$.
Since the ordinary singularities as well as $\ba_{2k+1}$ are a product of smooth branches, this
forces the same local property $H_\varepsilon \cap \cC=\emptyset$ near $P$. The remaining two cases
are either simple vertical tangencies or $\ba_{2k}$. In the simple tangency case $y^2=x$, note that
locally in a ball $B_P$ around~$P$,
$$
\array{rl}
(H_\varepsilon \cap \cC)_P&=\{(u,v+\varepsilon \sqrt{-1})\in B_P \mid u,v\in \br, 
(v+\varepsilon \sqrt{-1})^2=v^2-\varepsilon^2+2v\varepsilon \sqrt{-1}=u\}\\
&=
\{(-\varepsilon^2,\varepsilon \sqrt{-1})\}.
\endarray
$$
Analogously, if $y^2=-x$, then $(H_\varepsilon \cap \cC)_P=\{(\varepsilon^2,\varepsilon \sqrt{-1})\}$.
The position of this obstruction point relative to the curve is depicted in Figure~\ref{fig:tangency-a2k}.

Finally, at an irreducible double singularity of type $\ba_{2k}$ of local equation $y^2=x^{2k+1}$ one 
can check that $(H_\varepsilon \cap \cC)_P=\{(-\varepsilon^{\frac{2}{2k+1}},\varepsilon \sqrt{-1})\}$,
where $\varepsilon^{\frac{2}{2k+1}}$ represents the only real $2k+1$ root of~$\varepsilon^2$. 
The position of this obstruction point relative to the curve is also depicted in Figure~\ref{fig:tangency-a2k}.
\end{obs}

\begin{figure}[ht]
\begin{center}
\begin{tikzpicture}[baseline=0]
\begin{scope}[shift={(6,0)},rotate=180]
\coordinate (O) at (0,0);
\coordinate (P) at (2,1);
\coordinate (Q) at (2,-1);
\node[above] at (O) {$P$};
\node[right=6pt] at (O) {$\star$};
\draw[fill] (O) circle [radius=.05];
\draw (P) to [out=-135,in=0] (O) to [out=0,in=135] (Q);
\end{scope}
\begin{scope}[shift={(1,0)}]
\coordinate (O) at (0,0);
\coordinate (P) at (2,1);
\coordinate (Q) at (2,-1);
\draw[fill=blue,opacity=.4] ($(O)+(.5,0)$) circle [radius=.5];
\node[below left=-2pt] at (O) {$P$};
\node[left=6pt] at (O) {$\star$};
\draw[fill] (O) circle [radius=.05];
\draw (P) to [out=-135,in=0] (O) to [out=0,in=135] (Q);
\end{scope}
\begin{scope}[shift={(-6,0)}]
\coordinate (O) at (0,0);
\coordinate (P) at (1,1);
\coordinate (Q) at (1,-1);
\node[right=2pt] at (O) {$P$};
\node[left=6pt] at (O) {$\star$};
\draw[fill] (O) circle [radius=.05];
\draw (P) to [out=180,in=90] (O) to [out=-90,in=180] (Q);
\end{scope}
\begin{scope}[shift={(-2,0)},rotate=180]
\coordinate (O) at (0,0);
\coordinate (P) at (1,1);
\coordinate (Q) at (1,-1);
\node[left=2pt] at (O) {$P$};
\node[right=6pt] at (O) {$\star$};
\draw[fill] (O) circle [radius=.05];
\draw (P) to [out=180,in=90] (O) to [out=-90,in=180] (Q);
\end{scope}
\end{tikzpicture}
\caption{Obstruction points at vertical tangencies and $\ba_{2k}$-singular points}
\label{fig:tangency-a2k}
\end{center}
\end{figure}
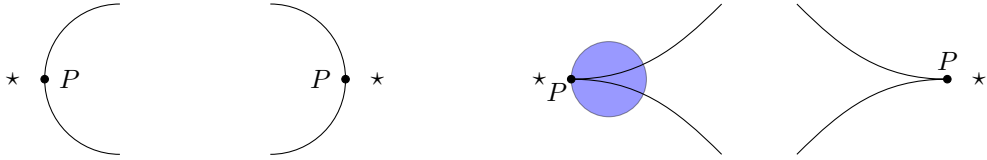


\begin{obs}
The compactness of~$B$ makes condition~\ref{thm:main2} in Theorem~\ref{thm:main0}
of a combinatorial nature because of the discussion in Remark~\ref{obs:local}.
\end{obs}

\begin{ejm}
It is straightforward to check that the curves $\cC:y^2-x^{m+1}=0$ are of Wirtinger type.
We apply Theorem~\ref{thm:main0} by
choosing $B$ as in Figure~\ref{fig:tangency-a2k}.
\end{ejm}

\begin{ejm}
As a simple application of Theorem~\ref{thm:main0} and Remark~\ref{obs:local}, note that the 
affine nodal cubic $\cC=\{y^2=x^2(x+1)\}$ is a curve of Wirtinger type. According to the discussion 
above, there is only an obstruction point (see Figure~\ref{fig:nodalcubic}) and the given $B$ and $L$ 
satisfy the conditions of Theorem~\ref{thm:main0}. Since the diagram $\cC_\br$ contains three edges
and only one vertex (associated with the nodal point $P$), $\G_\cC$ has three generators $x_1,x_2,x_3$ 
and only one set of relations as given in~\eqref{eq:ordinary}:
$$
x_2^2=x_3x_1, [x_2^2,x_1]=[x_2^2,x_2]=1, x_3=x_2,x_1=x_2,
$$
and hence $\pi_1(\bc^2\setminus \cC)=\bz$.

\begin{figure}[ht]
\begin{center}
\begin{subfigure}[b]{.45\textwidth}
\begin{center}
\begin{tikzpicture}[baseline=0,scale=.5]
\begin{scope}[shift={(0,0)}]
\coordinate (O) at (0,0);
\coordinate (P) at (2,4);
\coordinate (Q) at (2,-4);
\coordinate (R) at (-2,0);
\node[above] at (O) {$P$};
\node[left=6pt] at (R) {$\star$};
\draw[fill] (O) circle [radius=.05];
\draw (P) to [out=-100,in=35] (O) to [out=215,in=-90] (R) to [out=90,in=145] (O) to [out=-35,in=100] (Q);
\draw[fill opacity=.2,fill=blue] circle [radius=2];
\node[right=30pt] at (O) {$B$};
\node[above right=12pt] at (R) {$x_2$};
\node[above=2pt] at (.7,.7) {$x_1$};
\node[below=2pt] at (.7,-.7) {$x_3$};
\draw[dashed] (-1.25,-4)--(-1.25,4) node[below right]  {$L$};
\end{scope}
\end{tikzpicture}
\caption{Nodal cubic}
\label{fig:nodalcubic} 
\end{center}
\end{subfigure}
\begin{subfigure}[b]{.45\textwidth}
\begin{center}
\begin{tikzpicture}[scale=.6,baseline=0]
\begin{scope}[shift={(0,0)}]
\draw (-2,-2) -- (4,4);
\draw (-4,4) -- (2,-2);
\coordinate (O) at (0,0);
\coordinate (O1) at (0,1);
\coordinate (P) at (-2,2);
\coordinate (Q) at (2,2);
\coordinate (P1) at (-3,5);
\coordinate (Q1) at (3,5);
\draw[fill] (O) circle [radius=.05];
\draw[fill] (P) circle [radius=.05];
\draw[fill] (Q) circle [radius=.05];
\draw (O1) to [out=0,in=-135] (Q) to [out=45,in=-95] (Q1);
\draw (O1) to [out=180,in=-45] (P) to [out=135,in=-85] (P1);
\node[left] at (Q1) {$z_3$};
\node[right] at (P1) {$z_1$};
\node[below right] at (4,4) {$x_1$};
\node[below left] at (-4,4) {$y_1$};
\node[above] at (O1) {$z_2$};
\node[below right] at (1,1) {$x_2$};
\node[below left] at (-1,1) {$y_2$};
\node[left] at (-1,-2) {$x_3$};
\node[right] at (1,-2) {$y_3$};
\end{scope}
\end{tikzpicture}
\caption{Parabola and tangent lines}
\label{fig:T442}
\end{center}
\end{subfigure}
\caption{}
\end{center}
\end{figure}

\end{ejm}

\begin{ejm}
Consider the parabola $y=x^2$ together with two parallel lines as in Figure~\ref{fig:T442}.
The union of these irreducible components is an affine curve $\cC$ of Wirtinger type. Choosing 
as $B$ a big enough rectangle centered at the origin containing all singularities and a vertical 
line $L$ placed at the left-most edge of $B$, one can trivially check they satisfy the 
hypotheses of Theorem~\ref{thm:main0}. 


The following is a complete set of relations obtained from the diagram~$\cC_\br$, namely,
$$
\array{lll}
\begin{cases}
(z_1y_1)^2=(y_1z_1)^2\\
z_2=(y_1z_1)^{-1}z_1(y_1z_1)\\
y_2=z_1^{-1}y_1z_1
\end{cases}
&
\begin{cases}
x_3y_2=y_3x_2\\
y_2=y_3\\
[y_2,x_3]=1\\
x_2=y_2^{-1}x_3y_2=x_3
\end{cases}
&
\begin{cases}
(z_2x_2)^2=(x_2z_2)^2\\
z_3=(x_2z_2)^{-1}z_2(x_2z_2)\\
x_1=z_2^{-1}x_2z_2
\end{cases}
\endarray
$$
Using the relations $x=x_2=x_3$, $y=y_2=y_3$, $x_1=z_2^{-1}xz_2$, $z_1=yz_2y^{-1}$
$y_1=z_1yz_1^{-1}=yz_2yz_2^{-1}y^{-1}$, $z_3=xz_2x^{-1}$ the presentation $\G_\cC$
can be reduced to
$$
\pi_1(\bc^2\setminus \cC)=\langle x,y,z_2: xy=yx, (yz_2)^2=(z_2y)^2, (xz_2)^2=(z_2x)^2\rangle
$$
which is the presentation of the Euclidean Artin group~$(4,4,2)$.
\end{ejm}

\section{Proof of Theorem~\ref{thm:main0}}\label{sec:main-thm}

\begin{proof}[Proof of Theorem{\rm~\ref{thm:main0}}]
Let us consider a curve of Wirtinger type $\cC$, a topological disk $B$ and a vertical line $L=p^{-1}(t_0)$ 
satisfying the hypotheses. For simplicity, we assume $\varepsilon>0$. The strategy of the proof is to inductively
transform a Zariski-van Kampen presentation of $\pi_1(\bc^2\setminus\cC;P_0)$ into the Wirtinger presentation~$\G_\cC$.

Before we start, a general method to construct loops is described as follows. Let $\ell\in E_\cC$, $p_\ell\in E_\cC$ a 
smooth point, and $\Delta_\ell$ a disk of radius~$\varepsilon$ centered at $p_\ell$ and transversal to $\cC$. Let $q_\ell$ 
be the unique point in $\Delta_\ell\cap B_\varepsilon$. The meridian $\mu_\ell$ is defined taking a path $\rho_\ell$ in 
$B_\varepsilon$ from $P_0$ to $q_\ell$, running $\partial\Delta_\ell$ counterclockwise and coming back to $P_0$ via 
$\rho_\ell^{-1}$. A key remark is that this construction defines a unique meridian $\mu_\ell$ independently of the 
choice of~$\rho_\ell$ and $p_\ell$ by condition~\eqref{eq:Be}.

We will start with an appropriate Zariski-van Kampen presentation for a 
suitable base point $P_0$ on $L$.
Let us write $L\cap\cC=L_\br\cap \cC_\br=\{p_{\ell_1}=(t_0,y_1),\dots,p_{\ell_d}=(t_0,y_d)\}$ where $p_{\ell_i}$ is a
smooth point in $\ell_i\in E_\cC$ with $y_1>\dots>y_d$ and choose $y_0\in \br$, $y_0\geq y_1$ such that 
$(t_0,y_0)\in B\cap L_\br$. This is possible since $L_\br\cap B$ is an interval, $L_\br\cap \cC_\br$ is a finite set 
of points, $B\cup L_\br\cup \cC_\br$ is simply connected, by condition~\ref{thm:main1}, and hence $L_\br\cap \cC_\br\subset B$. 
The point $P_0=(t_0,y_0+\varepsilon\sqrt{-1})\in B_\varepsilon\cap L$ will be taken as a base point.

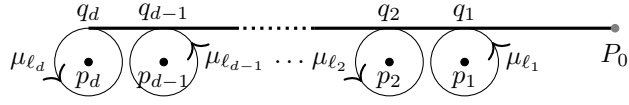
\begin{figure}[ht]
\begin{center}
\begin{tikzpicture}[vertice/.style={draw,circle,fill,minimum size=0.1cm,inner sep=0}]
\tikzset{%
  suma/.style args={#1 and #2}{to path={%
 ($(\tikztostart)!-#1!(\tikztotarget)$)--($(\tikztotarget)!-#2!(\tikztostart)$)%
  \tikztonodes}}
} 
\tikzset{flecha/.style={decoration={
  markings,
  mark=at position #1 with  {\arrow[scale=2]{>}}},postaction={decorate}}}
\def\radio{.45}
\coordinate (A) at (3,0);
\coordinate (B) at (1,0);
\coordinate (C) at (0,0);
\coordinate (D) at (-1,0);
\coordinate (E) at (-2,0);
\coordinate (F) at (-3,0);
\coordinate (G) at (-4,0);
\coordinate (Y) at (0,1);

\draw[line width=1.2] ($(A)+\radio*(Y)$)--($(D)+\radio*(Y)$);
\draw[line width=1.2,dotted] ($(D)+\radio*(Y)$)--($(E)+\radio*(Y)$);
\draw[line width=1.2] ($(E)+\radio*(Y)$)--($(G)+\radio*(Y)$);
\draw[flecha=.1] (B) circle [radius=\radio];
\draw[flecha=.6] (C) circle [radius=\radio];
\draw[flecha=.1] (F) circle [radius=\radio];
\draw[flecha=.6] (G) circle [radius=\radio];

\node[above=12pt] at (B) {$q_1$};
\node[above=12pt] at (C) {$q_2$};
\node[above=12pt] at (F) {$q_{d-1}$};
\node[above=12pt] at (G) {$q_d$};

\node[vertice,color=gray] at ($(A)+\radio*(Y)$) {};
\node[right,above=-5pt] at (A) {$P_0$};
\node[vertice] at (B) {};
\node[right=12pt] at (B) {$\mu_{\ell_1}$};
\node[below] at (B) {$p_1$};
\node[vertice] at (C) {};
\node[left=12pt] at (C) {$\mu_{\ell_2}$};
\node[below] at (C) {$p_2$};

\node[] at ($.7*(D)+.3*(E)$){$\dots$};

\node[vertice] at (F) {};
\node[right=12pt] at (F) {$\mu_{\ell_{d-1}}$};
\node[below] at (F) {$p_{d-1}$};
\node[vertice] at (G) {};
\node[left=12pt] at (G) {$\mu_{\ell_d}$};
\node[below] at (G) {$p_{d}$};
\end{tikzpicture}
\caption{Generators in the fiber}
\label{fig:base_fibra}
\end{center}
\end{figure}

As was mentioned above, the idea of the proof is to transform the Zariski-van Kampen presentation of 
$\pi_1(\bc^2\setminus \cC)$ into the Wirtinger presentation $\G_\cC$. To simplify this procedure one can transform 
slightly the Wirtinger presentation by considering an extended diagram where the vertices in $\tilde\cV_\cC$ contain 
$\cV_\cC$ and the vertical tangencies of $\cC$ and considering such points as $\ba_0$-singular points. 
The resulting relation is provided in~\eqref{eq:double} for $m=0$, i.e. the two generators coincide. 
The set of resulting edges will be denoted by~$\tilde E_\cC$.

Let us start from the Zariski-van Kampen presentation $\G_0$ of Corollary~\ref{cor:zvk} generated by 
the meridians
$\mu_{\ell_1},\dots,\mu_{\ell_d}$ as in Figure~\ref{fig:base_fibra},
where $\ell_j$ is the edge containing $p_j$. Recall that the relators in $\G_0$ correspond 
to the singular points of the projection~$p|_\cC$, i.e.~with the vertices of the modified Wirtinger presentation 
$\tilde\G_{\cC}$. Let us order the set $\tilde\cV_\cC=\{P_1,\dots,P_r\}$ of singular points of the projection~$p|_\cC$ 
by its distance to~$L_\br$. Denote by $L_j$, $j=1,\dots,r$ (resp.~$L_0$) the vertical line 
containing $P_j$ (resp.~$L_\br$).
An inductive procedure will be presented 
to transform $\G_0$ into $\G_r=\tilde\G_\cC$ using only Tietze transformations of type~I and IIa (without IIb~\cite{dun:76}, i.e. the homotopy type is preserved).
At each step $j\in\{0,\dots,r\}$, a presentation will be given whose generators are associated with the edges in  
$\tilde E_\cC \cap (L_0\cup\dots\cup L_j) $ 
and whose relations associated with $P_1,\dots,P_j$ coincide with those 
of~$\tilde\G_\cC$ while the ones associated with
$P_{j+1},\dots,P_r$ are still those of Zariski-van Kampen presentation.


For $j=0$, the result is trivial using $\G_0$. Assume $\G_j$ is constructed and consider the point $P_{j+1}$
and its associated relations. The only new edges in 
$\tilde E_\cC\cap 
L_{j+1}$ might come from adjacent edges to $P_{j+1}$.
If $P_{j+1}$ is of type $\ba_{2k}$ and since $P_{j+1}$ faces $L_\br$ (by condition~\ref{thm:main3}), no new edges arise.
Let $x_{\ell'}$ be any generator associated with an edge~$\ell'$ adjacent to $P_{j+1}$, then $x_{\ell'}=x_{\ell}^{\eta_j}$ for some generator 
$x_{\ell}$ in $\G_0$, see~Corollary~\ref{cor:zvk}. 
The local braid $\delta_j$ described before Corollary~\ref{cor:zvk} is 
$\sigma_1^{2k+1}$. Hence relation~\ref{R2} produces $x_{\ell'}=x_{\ell'}^{\delta_j}$, which becomes 
$x_{\ell}^{\eta_j}=x_{\ell}^{\delta_j\eta_j}$, 
that is, 
the Zariski-van Kampen relation associated with~$P_{j+1}$,
which is replaced by the Wirtinger relation~\ref{R2} in~$G_{j+1}$.

For the remaining cases ($\ba_{2k+1}$ and ordinary), there are new edges in $\tilde E_\cC$ adjacent to $L_{j+1}$ and the local braids 
$\delta_j$ are squares, say $\delta_j=\tilde\delta_j^2$. As above, the relations involving the Zariski-van Kampen relation can 
be analogously replaced in $\G_j$ by those in~\ref{R1} and~\ref{R2} involving only the old edges (denoted by $x$'s). 
As above let $x$ be any generator associated with an edge adjacent to $P_{j+1}$ and facing $L_\br$ and let $y$ be the corresponding 
new edge on the same irreducible component. We will further transform $\G_j$ by adding a new generator $y$ and a relation
$y=x^{\tilde\delta_j}$ (see Remark~\ref{rem:xrels}). This process is continued until $\G_r=\tilde\G_\cC$ is obtained.
\end{proof}

\begin{obs}
As a consequence of the beginning of the proof, a homomorphism $h:\G_\cC\to \pi_1(\bc^2\setminus \cC)$ 
can be defined as follows. Given $x_\ell$ the generator of $\G_\cC$ corresponding to~$\ell\in E_\cC$ as in 
\S\,\ref{sec:wirtinger}, then $h(x_\ell):=\mu_\ell$. The rest of the proof shows that $h$ is in fact an isomorphism.
\end{obs}

\begin{cor}
The Wirtinger presentation of a curve of Wirtinger type satisfying the conditions of Theorem~\ref{thm:main0} has
the homotopy type of its associated CW-complex.
\end{cor}

\begin{proof}
Since Tietze transformations of type III are used in the proof of Theorem~\ref{thm:main0}, this corollary 
is a consequence of the proof of the main result in~\cite{Libgober-homotopytype}, where the transversality 
with the line at infinity is not needed in his own proof. Incidentally, despite all the strong genericity conditions
stated in his result, only the non-existence of vertical asymptotes is actually required.
\end{proof}

The conditions of Theorem~\ref{thm:main0} are sufficient, but not necessary. 
The following examples illustrate that the conditions
in the statement of this theorem are not only technical.

\begin{ejm}
Let~$\cC$ be the curve defined by $f(x,y)=y^{3}-  y^{2} + 10 x^{2} y + x^{3}$.
A sketch of its real picture is in Figure~\ref{fig:cusp_tan}. Note that  condition
\ref{thm:main3} in Theorem~\ref{thm:main0} is not fulfilled. It is straightforward to see that 
$\G_\cC=\langle x_1,x_2\mid x_1\cdot x_2\cdot x_1=x_2\cdot x_1\cdot x_2\rangle$ while 
$\pi_1(\bc^2\setminus\cC)\cong\bz$ and hence $\pi_1(\bc^2\setminus\cC)$ and $\G_\cC$
are not isomorphic.
\end{ejm}

\begin{ejm}\label{ejm:cardiode}
Let~$\cC$ be the cardioid curve~$\cC$ defined by $f(x,y)=(y^2+x^2-2x)^2- 4(x^2+y^2)$,
see Figure~\ref{fig:cardioide}. It is not possible to find a simply connected
region~$B$ satisfying the hypotheses of Theorem~\ref{thm:main0}
(because of the obstruction point close to the cusp). It is straightforward to see that 
$\G_\cC\cong\bz$. The projective curve is the tricuspidal  quartic curve
(with two cusps at infinity), which has a non-abelian fundamental group (as proved in~\cite{zr:29}).

\begin{figure}[ht]
\begin{center}
\begin{subfigure}[b]{.45\textwidth} 
\begin{center}
\begin{tikzpicture}[scale=.8,baseline=0,vertice/.style={draw,circle,fill,minimum size=0.15cm,inner sep=0}]
\draw[fill=blue,fill opacity=.2,] (1.1,-1)--(1.1,2)--(-2.1,2)--(-2.1,-1)--(-.9,-1)--(-.9,1)--(-.1,1)--(-.1,-1)--cycle;
\draw (2,-1) to[out=135,in=0] (0,0)
to[out=0,in=-90] (1,1)
to[out=90,in=90] (-2,1)
to[out=-90,in=90] (-1,0)
to[out=-90,in=45] (-2.5,-1);
\node[right] at (1,1) {$\star$};
\node[left] at (0,0) {$\star$};
\node[right] at (-1,0) {$\star$};
\node[left] at (-2,1) {$\star$};
\node[vertice] at (0,0) {};
\node at (.5,-.3) {$x_2$};
\node at (.5,.5) {$x_1$};
\node at (0,-1.75) {};
\end{tikzpicture}
\caption{Cuspidal cubic generic at infinity}
\label{fig:cusp_tan}\end{center}
\end{subfigure}
\begin{subfigure}[b]{.45\textwidth} 
\begin{center}
\begin{tikzpicture}[scale=.8,baseline=0,vertice/.style={draw,circle,fill,minimum size=0.15cm,inner sep=0}]
\begin{scope}[xshift=1cm]
\draw[fill=blue,fill opacity=.2,] (1.1,-1.75)--(1.1,1.75)--(-2.1,1.75)--(-2.1,-1.75)--cycle;
\draw (1,-1) to[out=90,in=0] (0,0)
to[out=0,in=-90] (1,1)
to[out=90,in=90] (-2,0)
to[out=-90,in=-90] (1,-1);
\node[right] at (1,1) {$\star$};
\node[left] at (0,0) {$\star$};
\node[right] at (1,-1) {$\star$};
\node[left] at (-2,0) {$\star$};
\node[vertice] at (0,0) {};
\node at (.5,-.5) {$x_2$};
\node at (.5,.5) {$x_1$};
\end{scope}
\end{tikzpicture}
\caption{Cardioid}
\label{fig:cardioide}\end{center}
\end{subfigure}
\caption{}
\end{center}
\end{figure}

\end{ejm}

\section{Artin groups and Hypocycloids}\label{sec:hypocycloids}
In this section Artin groups and basic properties of hypocycloids will be recalled in order to state 
and prove the main Theorem on the fundamental group of hypocycloid $\cC_{k,k-1}$ curves. There are 
many conventions to define Artin groups, but for our purpose the Dynkin-diagram convention will be 
more suitable.

\begin{dfn}
Let $\Gamma$ be a graph, the Artin group associated to $\Gamma=(V,E)$ is defined as group 
$G_\Gamma$ generated by the vertices~$v\in V$ of $\Gamma$ with the following relations: 
\begin{equation}
\label{eq:artin}
v\cdot w\cdot v=w\cdot v\cdot w \text{ if } \{v,w\}\in E,\quad \text{ and } [v,w]=1 \text{ otherwise.}
\end{equation}
\end{dfn}

\begin{ejm}
Let $\G_N:=G_{\tilde A_N}$ be the Artin group of an $N$-gon (an affine Dynkin diagram $\tilde A_N$). 
According to~\eqref{eq:artin} a presentation of $\G_N$ can be written as
\[
\left\langle
x_j, j\in \bz_N
\mid 
x_i\cdot x_{i+1}\cdot x_i=x_{i+1}\cdot x_i\cdot x_{i+1}, \quad 
x_i\cdot x_j = x_j\cdot x_i \text{ for } |i-j|\neq 1
\right\rangle,
\]
where the subindices are considered in $\bz_N$ and $|i-j|\neq 1$ means $i-j\not\equiv \pm 1 \mod N$.

Consider $\bz/2=\langle t\mid t^2=1\rangle$ acting on $\G_N$ as $x_j^{t}:=x_{-j}$. 
In the special case $N=2k-1$, it is straightforward to check that the semidirect product 
$\G_N\rtimes\bz/2$ admits a presentation generated by $t$, $x_0,\dots,x_{k-1}$ and whose relations are:
\begin{enumerate}
\enet{(SD\arabic{enumi})} 
\item\label{sd1} $t^2=1$;
\item $x_j\!\cdot\! x_{j+1}\!\cdot\! x_j\!=\!x_{j+1}\!\cdot\! x_j\!\cdot\! x_{j+1}$ for
$0\leq j<k-1$;
\item $(x_{k-1}\!\cdot\! t)^3\!=\! (t\!\cdot\! x_{k-1})^3$;
\item $[x_0,x_j]=1$ for $1<j\leq k-1$;
\item $[x_i,x_j]=1$ for $0<i,j\leq k-1$ and $j-i>1$;
\item\label{sd6} $[x_i,t\cdot x_j\cdot t]=1$ for $0<i\leq j\leq k-1$ and $(i,j)\neq (k-1,k-1)$.
\end{enumerate}

\end{ejm}


Let us recall some of the main properties of the hypocycloids. We follow the notation of~\cite{ac} 
--\,see~\cite{morley:91} for details too. A hypocycloid is a real curve associated with each pair $k,\ell\in\bz$ 
such that $0<\ell<k$ and $\gcd(k,\ell)=1$. Consider $N=k+\ell$ and any pair of positive integers $r$, $R$ 
such that $\frac{r}{R}=\frac{\ell}{N}$ (or $\frac{k}{N}$), the hypocycloid is the real curve obtained as 
the trace of a fixed point on a circumference of radius $r$ when rolling inside a circumference of radius~$R$.
This real curve admits an algebraic equation and thus one can consider its complexification $\cC_{k,\ell}$ 
as the complex curve in $\bc^2$ defined by this algebraic equation.


The complex curve $\cC_{k,\ell}$ is rational and has degree~$2k$ and its projective closure 
contains two points at infinity --\,the so-called concyclic points. 
As a summary of its algebraic properties:
\begin{enumerate}
\enet{(C\arabic{enumi})} 
\item $\cC_{k,\ell}$ has $N$ ordinary cuspidal singular points.
\item $\cC_{k,\ell}$ has $N(k-2)$ ordinary double points. In its classical presentation,
$N(\ell-1)$ of them are real points with real tangent lines while the other $N(k-\ell-1)$ are 
complex.
\item The two points at infinity of $\bar{\cC}_{k,\ell}$ have local equations topologically
equivalent to $u^{k-\ell}+v^{k}=0$ (tangent to the line at infinity with contact order~$k$).
\end{enumerate}

These properties are classical and imply that the curve is rational, see e.g.~\cite{ac} for a modern 
description, precise formul{\ae}, and parametrizations.

Let us consider equations such that a point in the real axis has a vertical tangency. 
Note that $p_{|\cC_{k,\ell}}:\cC_{k,\ell}\to\bc$ is a proper $2k$-fold branched covering, extending 
to $\bar{p}_|:\bar{\cC}_{k,\ell}\to\bp^1\equiv\mathbb{C}\cup\{\infty\}$,
where $\bar{\cC}_{k,\ell}$ is the normalization of its projective closure.
Since the curve is rational, $\bar p_|$ has $2(2k-1)$ points of ramification, counted with multiplicity. 
Two of them lie at~$\infty$, each one with multiplicity~$k-1$. Since no tangent line to the cuspidal 
points is vertical, each one of the $N$ cusps contributes with multiplicity~$1$. 
The remaining multiplicity accounts for the amount simple tangencies of the projection, namely,
\[
2(2k-1)-2(k-1)-(k+\ell)=k-\ell.
\]
Whenever $N$ (and hence $k-\ell$) is even only two of them are real, while in the odd case exactly one 
is real, leaving $k-\ell-1$ of them in pairs of complex conjugated tangencies. Thus, in the special 
case $\ell=k-1$ there must be only one vertical non-transversal line, which has to be real, $N$~lines through
the cusps, and the vertical lines passing through the nodes (all of them real). Since the horizontal 
axis is a symmetry axis, it intersects the curve --\,tangentially\,-- at a cusp, at the $k-2$ nodes, and
at the point of vertical tangency. Hence the nodes are in $k-2+\frac{N-1}{2}(k-2)=(k-2)k$ vertical lines.
Despite the fact that the real picture of $\cC_{k,k-1}$ contains all the topological information of its 
embedding in~$\bc^2$, Theorem~\ref{thm:main0} cannot be applied directly since it does not 
satisfy~\ref{thm:main3}. However its quotient by the horizontal reflection will.

\begin{figure}[ht]
\begin{center}
\begin{tikzpicture}[scale=.8,vertice/.style={draw,circle,fill,minimum size=0.2cm,inner sep=0}]
\begin{scope}[scale=2]
\def\xk{2}
\def\xl{1}
\draw [line width=1pt] plot [parametric,domain=0:2*pi,samples=100]
        ({(\xk*cos(deg(\xl*\x))+\xl*cos(deg(\xk*\x)))/(\xk+\xl)},
{(\xk*sin(deg(\xl*\x))-\xl*sin(deg(\xk*\x)))/(\xk+\xl)});
\end{scope}
\begin{scope}[xshift=4cm,scale=2]
\def\xk{3}
\def\xl{2}
\draw [line width=1pt] plot [parametric,domain=0:2*pi,samples=100]
        ({(\xk*cos(deg(\xl*\x))+\xl*cos(deg(\xk*\x)))/(\xk+\xl)},
{(\xk*sin(deg(\xl*\x))-\xl*sin(deg(\xk*\x)))/(\xk+\xl)});
\end{scope}
\begin{scope}[xshift=8cm,scale=2]
\def\xk{4}
\def\xl{3}
\draw [line width=1pt] plot [parametric,domain=0:2*pi,samples=100]
        ({(\xk*cos(deg(\xl*\x))+\xl*cos(deg(\xk*\x)))/(\xk+\xl)},
{(\xk*sin(deg(\xl*\x))-\xl*sin(deg(\xk*\x)))/(\xk+\xl)});
\end{scope}

\end{tikzpicture}

\caption{Curves $\cC_{k,k-1}$ for $k=2,3,4$.}
\label{fig:hypo}\end{center}
\end{figure}

The equation $f(x,y)=0$ of $\cC_{k,k-1}$ satisfies $f(x,y)=g(x,y^2)$
for some $g\in\bc[x,y]$. The curve $\cD_{k,k-1}$ 
defined by $h(x,y):=y g(x,y)$ will do the trick.

\begin{lema}
The curve $\cD_{k,k-1}$ satisfies the hypotheses of Theorem{\rm~\ref{thm:main0}}.
\end{lema}

\begin{proof}
Note that the projection change the local type of branches intersecting the horizontal line,
namely, it converts nodes into tangencies, branches with vertical tangencies into transversal 
branches with respect to the vertical line, and 
cusps into inflection points, see Figure~\ref{fig:cociente}. A topological disk~$B$ can be chosen
as in Figure~\ref{fig:cociente}.
The line~$L$ is close to the node in the horizontal axis.
Since there is no vertical tangency and all the cusps are facing outwards, the result follows.
\end{proof}

Let us label the edges of the non-linear component $\cD$ of the curve~$\cD_{k,k-1}$.
Its real part $\cD_\br$ is the image of a map $\br\to\br^2$,
starting from the negative $x$-half plane (to the image
of the transversal intersection to the horizontal line)
to the part after the inflection point.

In order to label the edges we follow some conventions. First, we do not change
the labels when passing through a node. We start with the label $x_0$ for the 
edge transversal to the horizontal line. We continue as follows:
\begin{enumerate}
\item If $x_j$ ends in a tangency in the horizontal line then the next edge is $x_{-j}$.
 \item If $x_{\varepsilon j}$, $j\geq 0$, $\varepsilon=\pm 1$ (assume $\varepsilon=(-1)^{k}$ for $j=0$), then
 the next edge is $x_{\varepsilon(j+1)}$
\end{enumerate}
With these conventions, the inflection edge is $x_{k-1}$.
This procedure is illustrated in Figure~\ref{fig:cociente} 
for $k=4$.

Actually, we are not interested in the group $\pi_1(\bc^2\setminus\cD_{k,k-1})=\G_{\cD_{k,k-1}}$ but
in its quotient $G_k$ obtained by~\emph{killing} the square of a meridian of the horizontal line.
This is an orbifold fundamental group. If we consider the double cover we obtain
that $\pi_1(\bc^2\setminus\cC_{k,k-1})$ is the kernel of the epimorphism 
$G_k\to\bz/2$ which sends the meridians of the line to~$1$ and the meridians of the other component
to~$0$.

\begin{prop}
The group~$G_k$ is isomorphic to the semidirect product $\G_N\rtimes\bz/2$,
where $\G_N$ is the Artin group of the~$N$-gon.
\end{prop}

\begin{proof}
Note that the group $G_k$ is generated by $x_{2-k},\dots,x_{k-1}$ (the edges of the quotient
of the hypocycloid) and~$t$, which is the generator in the horizontal line corresponding
to the edges adjacent to the normal crossing of the hypocycloid ($x_0$).
Note that all the tacnodes to the left of~$t$ correspond to $(-1)^k$-labels, while the tacnodes
located to the right of~$t$ correspond to $(-1)^{k+1}$-labels.

Let $x_j$, $j\neq 0,k-1$, be an edge of a tacnode in the side closer to~$t$. Then,
$x_{-j}=t'\cdot x_j\cdot t'$, and $t'$ is obtained conjugating~$t$ by a product of
$x_i$'s, where all the indices~$i$ are of the same parity as~$j$ (and distinct form $k-1$).
As a consequence, $x_{-j}=t\cdot x_j\cdot t$
and we can check that we obtain the presentation of $\G_N\rtimes\bz/2$ generated by 
$t,x_0,\dots,x_{k-1}$ with relations~\ref{sd1}-\ref{sd6}.
\end{proof}

\begin{figure}[ht]
\begin{center}
\begin{tikzpicture}[scale=1,vertice/.style={draw,circle,fill,minimum size=0.1cm,inner sep=0}]
\def\ang{360/7}
\coordinate (A1) at (2,0);
\coordinate (A2) at ($({3*cos(1*\ang)},{2*sin(1*\ang)})$);
\coordinate (A3) at ($({1*cos(2*\ang)},{2*sin(2*\ang)})$);
\coordinate (A4) at ($({2*cos(3*\ang)},{2*sin(3*\ang)})$);
\coordinate (X1) at (1.25,0);
\coordinate (X2) at (0.25,0);
\coordinate (X3) at (-.75,0);

\coordinate (P0) at (1.5,.25);
\coordinate (P1) at (.4,.3);
\coordinate (P2) at (.6,.6);
\coordinate (P3) at (.5,.5);
\coordinate (P4) at (-.5,.7);

\draw[fill=blue,opacity=.2] ($(A1)+(1,-0.5)$)-- ($(A1)+(1,0)$)--(A2)--(A3)--(A4)--($-1*(A1)$)--($-1*(A1)-(0,0.5)$)--cycle;

\def\ci{\draw[name path=c1] ($(A1)+(1,-.2)$) to[out=150,in=0] (A1) 
 to[out=180,in=-20] (P0)
 to[out=160,in=-10] (P1)
to[out=170,in=-10] (P4)
to[out=170,in=3*\ang-180] (A4)}

\def\cii{\draw[name path=c2](A4)
to[out=3*\ang-180,in=180] (X3)
to[out=0,in=-135] (P3)
to[out=45,in=\ang+180] (A2)}

\def\ciii{\draw[name path=c3] (A2)
to[out=\ang+180,in=0] (X1)
}

\def\civ{\draw[name path=c4] (X1)
to[out=180,in=-70] (P2)
to[out=110,in=180+2*\ang] (A3)
}

\def\cv{\draw[name path=c5]
(A3)
to[out=180+2*\ang,in=135] (X2)
to[out=-45,in=135] ($(X2)+(1,-1)$)}
\ci;
\cii;
\ciii;
\civ;
\cv;

\draw[line width=1.2] (-2,0)--(3,0) ;
\draw[dashed] (.015,-1)--(.015,1.75) ;

\path [name intersections={of=c1 and c3,by=N1}];
\path [name intersections={of=c1 and c4,by=N2}];
\path [name intersections={of=c1 and c5,by=N3}];
\path [name intersections={of=c2 and c5,by=N4}];
\path [name intersections={of=c1 and c2,by=N5}];
\path [name intersections={of=c2 and c4,by=N6}];

\node[vertice] at (A1) {};
\node[vertice] at (A2) {};
\node[vertice] at (A3) {};
\node[vertice] at (A4) {};
\node[vertice] at (X1) {};
\node[vertice] at (X2) {};
\node[vertice] at (X3) {};
\node[vertice] at (N1) {};
\node[vertice] at (N2) {};
\node[vertice] at (N3) {};
\node[vertice] at (N4) {};
\node[vertice] at (N5) {};
\node[vertice] at (N6) {};

\node[below left] at (A3) {$x_0$};
\node[right=3pt] at (A3) {$x_1$};
\node[below right] at (A2) {$x_{-1}$};
\node[above left=-2pt] at (A2) {$x_{-2}$};
\node[below] at (A4) {$x_2$};
\node[right=3pt] at (A4) {$x_3$};
\node[below right=3pt] at (X3) {$t$};


\end{tikzpicture}

\caption{Quotient curve of $\cC_{4,3}$.}
\label{fig:cociente}
\end{center}
\end{figure}

\begin{cor}
\label{cor:hypocycloids}
The fundamental group of the complement of hypocycloids $\cC_{k,k-1}$
is the Artin group of the polygon in $N=2k-1$ vertices.
In particular, $\G_{\cC_{k,k-1}}$ is a Wirtinger presentation of $\pi_1(\bc^2\setminus\cC_{k,k-1})$.
\end{cor}

\section{Extending the method}\label{sec:wirtingerext}

There are several ways to improve this method to compute fundamental
groups of complements of groups. The list of allowed singularities
can be enlarged. Besides ordinary singular points (without vertical
tangencies), any singular point where all the branches are
real and smooth is allowed (as far as no branch has vertical tangency).
As in~\ref{R1}, we will add the local relations induced by
Corollary~\ref{cor:zvk} and one relation for each branch in order
to express the generators of one side in terms of the generators
of the other side.

There are other allowed singular points besides ordinary and double points. Namely singular points
whose local irreducible components are all double ($\ba_{2m}$),
real and must be on the same side with respect to the vertical line
(and transversal to it).
In fact, in this case, we can admit smooth (real) branches with
vertical tangency, such that the intersection number with the vertical
line is~$2$ (for each branch); note that we may allow larger
intersection multiplicities between these smooth branches.
If we want these curves
to match with Theorem~\ref{thm:main0}, they must satisfy
the \emph{facing} condition~\ref{thm:main3}.

Finally, even vertical lines can be admitted as far as the following
condition is fulfilled. 
\begin{enumerate}
\item The global intersection number of the vertical line and $\cC_\br$
 equals to $\deg_y\cC$.
\item If one branch in~$L$ is smooth and transversal to the vertical line, then all branches are.
\item If one branch in~$L$ has intersection number~$2$ with $L$,
the same arises for the other branches and all of them are in the
same side. Moreover, they must satisfy condition~\ref{thm:main3}
in order to apply Theorem~\ref{thm:main0}.
\end{enumerate}

For the cases where Theorem~\ref{thm:main0}\ref{thm:main3}
cannot be avoided there are several ways to provide
a combinatorial (and correct) presentation of $\pi_1(\bc^2\setminus\cC)$. One of them uses the the real part of the pairs of 
complex conjugate branches, but their computation may be involved,
see~\cite{acct:01}.

If we use the ideas of the proof of Theorem~\ref{thm:main0},
we can still recover a combinatorial description of the fundamental
group. In order to do this, if $B\cap\cC\neq\emptyset$, the construction
of the loops associated to each edge of $\cC_\br$ must be done
taking into account this intersection points. Moreover,
the relations must involved the right loops.

\begin{obs}\label{obs:orientacion}
Let us describe the loops associated with one of these
intersection points, namely one close to a 
a point of type $\mathbb{A}_{2m}$.
Let us assume that in local coordinates it
has local equation $y^2+\alpha x^{2m+1}=0$, $\alpha^2=1$. 
Following the computations in Remark~\ref{obs:local}, its intersection with
$H_\varepsilon$ is $(\alpha\eta,\sqrt{-1}\varepsilon)$
where $\eta=\varepsilon^{\frac{2}{2m+1}}$. Let us check if the intersection
is transversal and which is the intersection number if the curve
is naturally oriented as a complex curve and if $H_\varepsilon$
has counterclockwise orientation.

\begin{figure}[ht]
\begin{center}
\begin{tikzpicture}[scale=.75,vertice/.style={draw,circle,fill,minimum size=0.2cm,inner sep=0}]
\node at (-2,1) {$\mathbb{R}^2$};
\draw (-2,.5) node[left] {$x_1$} to[out=-30,in=180] (0,0) 
to[out=180,in=30] (-2,-.5) node[left] {$x_2$};

\draw[dashed] (-1,1)--(-1,-1) node[below left] {$x=-\eta$};
\draw[dashed] (0,1) node[above] {{$\mathbb{A}_{2m}$}}--(0,-1)
node[below] {$x=0$};
\draw[dashed] (1,1)--(1,-1) node[below right] {$x=\eta$};

\begin{scope}[xshift=5cm]
\draw (-1,0)--(-2,-1) --
node[below] {$x=-\eta$}
(2,-1) --(3,0) ;
\draw (-1.3,-.5)--(2.3,-.5) node[above] {$\mathbb{R}$};
\node [vertice] at (-.5,-.5) {};
\node [below left] at (-.5,-.5) {$-\varepsilon$};
\node [vertice] at (1.5,-.5) {};
\node [below left] at (1.5,-.5) {$\varepsilon$};
\draw[->](3.2,-1)--(3.2,1.5) node[right,pos=.5] {$x=-\eta e^{\sqrt{-1}\pi t}$} ;
\node at (3,-1.2) {Braid:$\sigma_1^m$};
\end{scope}

\begin{scope}[xshift=5cm, yshift=1.5cm]
\draw (-1,0)--(-2,-1) --
node[above left] {$x=\eta$}
(2,-1) --(3,0) ;
\draw (0,-.9)--(1,.2) node[above] {$\sqrt{-1}\mathbb{R}$};
\node [vertice] at ($.8*(0,-.9)+.2*(1,.2)$) {};
\node [right] at ($.8*(0,-.9)+.2*(1,.2)$)  {$-\sqrt{-1}\varepsilon$};
\node [vertice] at ($.2*(0,-.9)+.8*(1,.2)$)  {};
\node [right] at ($.2*(0,-.9)+.8*(1,.2)$)  {$\sqrt{-1}\varepsilon$};
\end{scope}
\begin{scope}[yshift=3cm,yscale=.5]
\node at (-2,1.5) {$H_\varepsilon
$};

\draw[dashed] (-1,1)--(-1,-1) node[below left] {$x=-\eta$};
\draw[dashed] (0,1) -- (0,-1)
node[below] {$x=0$};
\draw[dashed] (1,1)--(1,-1) node[below right] {$x=\eta$};
\node[vertice] at (1,.5) {};
\node[right] at (1,.5) {$\sqrt{-1}\varepsilon$};
\end{scope}
\end{tikzpicture}

\caption{Meridian at $B_\varepsilon$}
\label{fig:epsilon}
\end{center}
\end{figure}
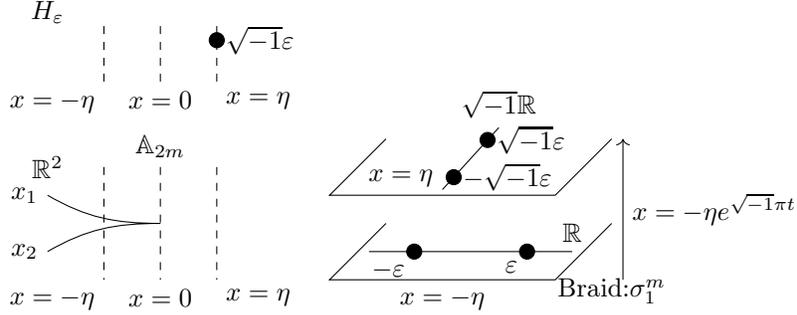

An $\br$-basis of the tangent plane for $H_\varepsilon$
is given by $\{(1,0),(0,\sqrt{-1})\}$. A $\bc$-basis for the complex tangent
line to the curve is given by $\{(2\sqrt{-1}\varepsilon),-(2m+1)\alpha\eta^{2m}\}$.
The natural orientation of this plane is given by completing the basis with
$\{(-2\varepsilon,-(2m+1)\alpha\eta^{2m}\sqrt{-1})\}$. 
If $\tilde{\eta}=(2m+1)\eta^{2m}$
the orientation is given by the sign of
\[
\det
\begin{pmatrix}
1&0&0&-2\varepsilon\\
0&0&2\varepsilon&0\\
0&1&-\alpha\tilde{\eta}&0\\
0&0&0&-\alpha\tilde{\eta}
\end{pmatrix}=
-\det\begin{pmatrix}
2\varepsilon&0\\
0&-\alpha\tilde{\eta}
\end{pmatrix}
=2\varepsilon\alpha\tilde{\eta}\alpha,
\]
which is the sign of~$\alpha$. Hence, if the cusp if left-sided, the sign is positive,
while right sided is negative.
What happens with the meridians of the intersections with $B_\varepsilon$?
In the situation of Figure~\ref{fig:epsilon}, the local loop
in~$B_\varepsilon$ is $y_1:=x_1^{(\sigma_1^{-m})}$; 
this
is related to the half-tour around the singularity from
$x=-\eta$ to $x=\eta$ (counterclockwise from $-\eta$ to $\eta$
and hence clockwise the other way around). If the
$\mathbb{A}_{2m}$-point is on the other side the loop is $y_1^{-1}$.
\end{obs}

Instead of describing a general procedure to compute this group,
let us apply it to a couple of examples.

\begin{ejm}
Let us consider the curve of Example~\ref{ejm:cardiode}.
We consider $L$ to the right of the cusp in Figure~\ref{fig:cardioide}
and we take the generators $x_1,x_2$ as in Figure~\ref{fig:base_fibra}.
The relation $x_1\cdot x_2\cdot x_1=x_2\cdot x_1\cdot x_2$ holds.
The bitangent vertical line provides no actual new relation, but the right-hand
side vertical tangency does. If we approach to it following the
loop $x_1$, in order to apply the local relation, we have to 
consider the loop $z\cdot x_2\cdot z^{-1}$, where $z$ is a counterclockwise loop 
around the obstruction point.
Following the discussion in Remark~\ref{obs:orientacion},
$z=y_1^{-1}$ and $y_1=x_1^{(\sigma_1^{-1})}=x_1^{-1}\cdot x_2\cdot x_1$.
Hence
\[
x_1=x_1^{-1}\cdot x_2^{-1}\cdot x_1\cdot x_2\cdot
x_1^{-1}\cdot x_2\cdot x_1\Longrightarrow
x_2\cdot x_1\cdot x_2^{-1}=x_1\cdot x_2\cdot
x_1^{-1}
\]
Applying the two relations we obtain:
\[
\pi_1(\bc^2\setminus\cC)=\langle x_1,x_2\mid x_1\cdot x_2\cdot x_1=x_2\cdot x_1\cdot x_2,[x_1^2,x_2]=1\rangle\cong
\bz\rtimes\bz/3.
\]
\end{ejm}

\begin{figure}[ht]
\begin{center}
\begin{tikzpicture}[baseline=0,vertice/.style={draw,circle,fill,minimum size=0.15cm,inner sep=0},scale=.7]
\begin{scope}[xshift=1cm]
\draw[fill=blue,fill opacity=.2,] (2,-2)--(2,2)--(-2,2)--(-2,-2)--cycle;
\draw (0,0) circle [radius=1];
\draw (0,0) circle [radius=2];
\draw[dashed] (-.8,-2.2)--(-.8,2.2);
\node[right] at (1,0) {$\star$};
\node[left] at (-1,0) {$\star$};
\node[right] at (2,0) {$\star$};
\node[left] at (-2,0) {$\star$};
\node[vertice] at (0,0) {};
\node at (-.45,-1.7) {$x_4$};
\node at (-.45,-.5) {$x_3$};
\node at (-.45,.5) {$x_2$};
\node at (-.45,1.6) {$x_1$};
\node[below left] at (-1,0) {$z_1$};
\node[below right] at (1,0) {$z_2$};
\end{scope}
\end{tikzpicture}

\caption{}
\label{fig:concentrica}
\end{center}
\end{figure}

\begin{ejm}
Let us consider the union of two concentric circumferences.
This curve is of Wirtinger type, see Figure~\ref{fig:concentrica}
for the vertical tangencies and note that they are tangent
at the concyclic points. It is clear
that $\G_\cC=\bz*\bz$.
The curve does not satisfy
the hypotheses of Theorem~\ref{thm:main0}.
If we choose the generators as in the proof of Theorem~\ref{thm:main0},
the tangencies of the inner curve, provide the equality $x_2=x_3$.

We denote by $z_1,z_2$ the counterclockwise loops around the inner
$\star$-points. From Figure~\ref{fig:concentrica}, we deduce
that $x_1=z_1\cdot x_4\cdot z_1^{-1}=z_2^{-1}\cdot x_4\cdot z_2$.
Using Remark~\ref{obs:orientacion}, we have
$z_1=x_2^{-1}$ and $z_2=x_2$. Hence, only the
relation $x_1=x_2^{-1}\cdot x_4\cdot z_2$ remains. As a consequence,
$\G_\cC$ is isomorphic to~$\pi_1(\bc^2\setminus\cC)$.

\begin{obs}
Using the extended Wirtinger method one can prove that the CW-complex associated with the 
Artin group $\G_{N}$ of the $N$-gon has the homotopy type of~$\bc^2\setminus \cC_{k,k-1}$. 
By simple Euler-characteristic calculations, the case $\ell=k-1$ is the only one where this can be expected.
\end{obs}
\end{ejm}

\bibliographystyle{amsplain}

\providecommand{\bysame}{\leavevmode\hbox to3em{\hrulefill}\thinspace}
\providecommand{\MR}{\relax\ifhmode\unskip\space\fi MR }
\providecommand{\MRhref}[2]{%
  \href{http://www.ams.org/mathscinet-getitem?mr=#1}{#2}
}
\providecommand{\href}[2]{#2}

\end{document}